\documentclass[12pt]{article}
\usepackage[T1]{fontenc}
\usepackage{amsmath, amsthm, amsfonts, mathrsfs, amssymb, setspace}
\usepackage[
pdftitle={On the graded automorphisms of upper triangular matrix algebras},
pdfdisplaydoctitle=true,
]{hyperref}

\newtheorem{Theo}{Theorem}

\newtheorem{Prop}{Proposition}
\newtheorem{Lemma}[Prop]{Lemma}
\theoremstyle{definition}
\newtheorem{Def}[Prop]{Definition}
\title{On the graded automorphisms of upper triangular matrix algebras}
\author{Felipe Y. Yasumura\thanks{Supported by PhD grant number 2013/22.802-1, from FAPESP}}
\date{September, 2017}

\begin{document}
\maketitle
\begin{abstract}
	We compute the graded automorphisms of the upper triangular matrices, viewed as associative, Lie and Jordan algebras. We compute also the so called self-equivalences and Weyl and diagonal groups for every grading.
\end{abstract}

\section{Introduction}

It is well known that every automorphism of a central simple associative algebra is inner. The same statement was proved to be true for the associative algebra of upper triangular matrices \cite{J1987,J1995}. Similar questions can be raised for algebras with additional structure, for example, in the context of graded algebras.

Recently, graded algebras are a subject of intense investigation, because its naturality in Physics and Mathematics. Polynomial algebras (in one or more commutative variables) are the most natural structure of an algebra with a grading - given by the usual degree of polynomials. For instance, the classification of finite dimensional semisimple Lie algebras gives rise to naturally $\mathbb{Z}^m$-graded algebras (see \cite{SanMartin,Jac1979}). Kemer solved a very difficult problem known as the Specht property in the theory of algebras with Polynomial Identities (PI-algebras, for short), in the setting of associative algebras over fields of characteristic zero, using $\mathbb{Z}_2$-graded algebras as a tool \cite{Kemer}. After the works of Kemer, the interest for graded algebras greatly raised.

In this short note, we shall investigate the structure of automorphisms, in the graded sense, for the algebra of upper triangular matrices. We will compute the graded automorphisms and the Weyl and diagonal groups of the upper triangular matrix algebras, as associative, Lie and Jordan algebras. We cite also the graded involutions on the associative case, computed firstly in \cite{VaZa2009}.

We denote by $G$ a group with multiplicative notation and neutral element $1$. Everything will be over an arbitrary field $K$.

We recall that an algebra $A$ is called $G$-graded (or $A$ is endowed with a $G$-grading) if there exists a vector space decomposition $A=\oplus_{g\in G}A_g$, where some $A_g$ can be null, such that $A_gA_h\subset A_{gh}$, for all $g,h\in G$. Sometimes we denote $\Gamma:A=\oplus_{g\in G}A_g$ to name $\Gamma$ the corresponding $G$-grading. A non-zero element $x\in A_g$ is called \textit{homogeneous} of $G$-degree $g$, and we denote $G\text{-}\deg x = g$ (or simple $\deg x=g$ whenever there is no ambiguity). We denote $\text{Supp}\,\Gamma=\{g\in G\mid A_g\ne0\}$ (or, if the grading over $A$ is clear, $\text{Supp}\,A$). If $B=\oplus_{g\in G}B_g$ is another $G$-graded algebra, we say that a linear map $f:A\to B$ is a \textit{homomorphism of graded algebras} if $f$ is an homomorphism of algebras and $f(A_g)\subset B_g$ for all $g\in G$. Similarly we define the notion of isomorphism of graded algebras (or graded isomorphism).

The notion of isomorphism of graded algebras is very natural. But we can also relate algebras graded by different groups, if we extent this definition. We introduce some of these extended notions. According to the book of Elduque and Kochetov \cite{EldKoc}, we consider the following sets, concerning the graded algebra $A$:
\begin{Def}[see \cite{EldKoc}]
	Let $\Gamma:A=\oplus_{g\in G}A_g$ be $G$-graded.
	\begin{enumerate}
		\renewcommand{\labelenumi}{(\roman{enumi})}
		\item $\text{Aut}(\Gamma)$ denotes all self-equivalences of $\Gamma$, that is, all isomorphism of algebras $f:A\to A$ such that for all $g\in\text{Supp}\,\Gamma$, there exists $h=\alpha(g)\in\text{Supp}\,\Gamma$ satisfying $f(A_g)=A_{h}$.
		\item $\text{Stab}(\Gamma)=\text{Aut}_G(A)$ is the set of all automorphisms of graded algebras of $A$.
		\item $\text{Diag}(\Gamma)$ is the set of all $f\in\text{Stab}(\Gamma)$ such that for all $g\in\text{Supp}\,\Gamma$, there exists $\lambda_g\in K$ such that $f(a)=\lambda_g a$, for all $a\in A_g$.
		\item $W(\Gamma)=\text{Aut}(\Gamma)/\text{Stab}(\Gamma)$ is the Weyl group of $\Gamma$.
	\end{enumerate}
\end{Def}
We remark that $\text{Stab}(\Gamma)$ is a normal subgroup of $\text{Aut}(\Gamma)$; indeed, every $f\in\text{Aut}(\Gamma)$ induces $\alpha\in\text{Sym}(\text{Supp}\,\Gamma)$ (the group of permutations of the set $\text{Supp}\,\Gamma$), in such a way $f(A_g)=A_{\alpha(g)}$. This defines a group homomorphism $f\in\text{Aut}(\Gamma)\to\alpha\in\text{Sym}(\text{Supp}\,\Gamma)$. Clearly $\text{Stab}(\Gamma)$ is exactly its kernel, hence $\text{Stab}(\Gamma)$ is indeed normal in $\text{Aut}(\Gamma)$. As consequence, $W(\Gamma)$ is actually a group.

It is well known that all these groups are related to the action of the Universal group grading (see bellow) on the algebra (see \cite{EldKoc}). We explicitly exhibits all these groups for the upper triangular matrices. But first, we briefly recall all possible group gradings over the upper triangular matrices, as associative, Lie and Jordan algebras (see \cite{VaZa2007,VinKoVa2004,pkfy1,pkfy2}, or the survey \cite{HKY}).\\[0cm]

\noindent\textbf{Group gradings on upper triangular matrices.} We denote by $UT_n$, $UT_n^{(-)}$ and $\text{UJ}_n$ the set of upper triangular matrices viewed as associative, Lie and Jordan algebra, respectively. The structure of associative algebra $UT_n$ is the usual product of upper triangular matrices; for the Lie algebra $UT_n^{(-)}$, we consider the bracket $[a,b]=ab-ba$; as Jordan algebra $\text{UJ}_n$, we take the product being the Jordan product $a\circ b=ab+ba$.

Let $e_{ij}$ be the matrix units, that is, $e_{ij}$ has value $1$ in the entry $(i,j)$ and $0$ in the other entries. From here on, $\ast$ can be either the associative, Lie or Jordan product on $UT_n$. We say that a grading on $(UT_n,\ast)$ is \textit{elementary} if all matrix units $e_{ij}$ are homogeneous in the grading. There are another two useful equivalents ways to work with elementary gradings.

The first equivalence is as follows. Every sequence $\eta=(g_1,g_2,\ldots,g_{n-1})\in G^{n-1}$ defines an elementary $G$-grading on the associative algebra $UT_n$ if we put $\deg e_{i,i+1}=g_i$ and $\deg e_{ii}=1$. Every elementary grading is given by such construction. If, moreover $G$ is commutative, this construction also defines an elementary grading on $\text{UJ}_n$ (or in $UT_n^{(-)}$). We denote this construction by $(UT_n,\eta)$ (or $(UT_n^{(-)},\eta)$, or $(\text{UJ}_n,\eta)$). Define the \textit{reverse sequence} of $\eta$ by $\text{rev}\,\eta=(g_{n-1},\ldots,g_2,g_1)$. We call an elementary grading \textit{symmetric} if $\eta=\text{rev}\,\eta$.

We describe another useful way of defining elementary grading. A $G$-grading on $UT_n$ is elementary if and only if there exists a sequence $a=(a_1,a_2,\ldots,a_n)\in G^n$ such that every $e_{ij}$ is homogeneous of degree $a_ia_j^{-1}$. We can always take $a_1=1$ (see \cite{VinKoVa2004}). For the Lie and Jordan case, we must have $[a_i,a_j]=1$, for all $i,j$.

Every group grading on $UT_n$ (as associative algebra) is graded isomorphic to an elementary one \cite{VaZa2007}. We mention also that in \cite{VinKoVa2004} the authors classified all elementary gradings on $UT_n$, and, for each elementary grading, they computed a basis for its graded polynomial identities.

Denote now $\odot$ the Lie bracket or the Jordan product on $UT_n$. We present a family of gradings on $(UT_n,\odot)$. Recall that $UT_n$ has an involution, given by $t:UT_n\to UT_n$ where $t(e_{ij})=e_{n-j+1,n-i+1}$; that is, $t$ satisfies $t^2=1$ and $t(ab)=t(b)t(a)$, for all $a,b\in UT_n$. If $H$ is an abelian group and $UT_n$ is equipped with an elementary $H$-grading where $\deg t(x)=\deg x$, for all homogeneous $x$, then we obtain a $G=\mathbb{Z}_2\times H$-grading on $(UT_n,\odot)$ if we consider the decomposition of $UT_n$ into homogeneous symmetric elements, and homogeneous skew-symmetric elements; this construction is an example of the family of gradings we define in what follows. Let $e_{i:m}=e_{i,i+m}$, $e_{-i:m}=e_{n-i+1-m,n-i+1}$ and $X_{i:m}^\pm=e_{i:m}\pm e_{-i:m}$. We say that a $G$-grading on $(UT_n,\odot)$ is \textit{MT} if all $X_{i:m}^\pm$ are homogeneous and $\deg X_{i:m}^+\ne\deg X_{i:m}^-$.

For the Jordan case, every group grading on $\text{UJ}_n$ is, up to a graded isomorphism, elementary or MT \cite{pkfy2}. Moreover, the support of any group grading on $\text{UJ}_n$ is commutative.

For the Lie case, the situation is a little bit more delicate. The center of $UT_n^{(-)}$, being non-trivial, can give rise to practically same gradings, but non-graded isomorphic. We introduce a notion to deal with the center (see, for instance, \cite{pkfy1}). Let $L$ be a Lie algebra and consider two $G$-gradings on $L$, say $L_1$ and $L_2$. We say that the gradings on $L_1$ and $L_2$ are \textit{practically the same} if the equality of graded algebras $L_1/\mathfrak{z}(L_1)=L_2/\mathfrak{z}(L_2)$ holds (we remark the easy assertion that the center of a graded Lie algebra is always a graded ideal), denoted $L_1\stackrel{G}= L_2$. If $L_1$ and $L_2$ are any $G$-graded Lie algebras, we say that $L_1$ is \textit{practically graded isomorphic} to $L_2$ if there exists a $G$-graded Lie algebra $L$ such that $L_1\simeq L$ (as graded algebras) and $L\stackrel{G}=L_2$; notation $L_1\stackrel{G}\sim L_2$. It is standard to prove that $\stackrel{G}\sim$ is an equivalence relation.

In \cite{pkfy1}, the authors prove that every group grading on the Lie algebra $UT_n^{(-)}$ is practically graded isomorphic to an elementary or MT one. It is also possible to prove that the support of an elementary or MT grading on $UT_n^{(-)}$ is always commutative. The unique possibility of obtaining non-commutative support for a grading on $UT_n^{(-)}$ is choosing adequate element to be the degree of the identity matrix; in other words, for any $G$-grading on $UT_n^{(-)}$, $\text{Supp}\,UT_n^{(-)}/\mathfrak{z}(UT_n^{(-)})$ is always commutative. Hence, we shall assume the grading group commutative when we deal with the Lie and Jordan cases.\\[0cm]

\noindent\textbf{Ordinary automorphism groups.} We will need also the well known group of automorphism of the upper triangular matrices as associative, Lie and Jordan algebras, in the ordinary sense. We review these groups.

Every automorphism of $UT_n$ is inner \cite{J1987,J1995}, that is, if $G_0=\{\varphi_a\mid a\in UT_n\text{ is invertible}\}$, where $\varphi_a:x\in UT_n\mapsto axa^{-1}\in UT_n$, then $\text{Aut}(UT_n)=G_0$. Moreover, after \cite{BBC}, every automorphism of $\text{UJ}_n$ is given by either an automorphism or antiautomorphism of $UT_n$. Hence $\text{Aut}(\text{UJ}_n)\simeq\langle t\rangle\rtimes G_0$ (where $t$ is the involution of $UT_n$, given by $t(e_{ij})=e_{n-j+1,n-i+1}$).

For the Lie case $UT_n^{(-)}$, we introduce an extra notation. Let $\omega=-t$, the minus flip along the second diagonal; it is easy to see that $\omega$ is an automorphism for $UT_n^{(-)}$. Denote by $S=\{(a_1,a_2,\ldots,a_n)\in K\mid a_1+a_2+\cdots+a_n+1\ne0\}$. For every $s=(a_1,\ldots,a_n)\in S$, let $\psi_s:UT_n^{(-)}\to UT_n^{(-)}$ be the map defined by $\psi(e_{ij})=e_{ij}+\delta_{ij}a_i1$, where $1$ is the identity matrix. Then $\psi_s$ is an automorphism of $UT_n^{(-)}$ \cite{Do1994}. Denote by $G_1=\{\psi_s\mid s\in S\}$. As proved in \cite{Do1994}, $G_0$ and $G_1$ commutes elementwise, $\omega$ normalizes $G_0\times G_1$ and $\text{Aut}(UT_n^{(-)})\simeq\langle\omega\rangle\rtimes(G_0\times G_1)$, if $n\ge3$; and $\text{Aut}(UT_2^{(-)})\simeq G_0\times G_1$.

\section{Associative case}

We will start computing the graded automorphisms for the associative case $UT_n$. Fix an elementary $G$-grading $\Gamma$ on $UT_n$. It is elementary to prove that $\deg e_{ii}=1$, for all $i$.

Let $a=(a_{ij})_{(i,j)}\in UT_n$ be an invertible matrix and write $a^{-1}=(b_{ij})_{(i,j)}$ its inverse. Denote by $\varphi_a:x\in UT_n\mapsto axa^{-1}\in UT_n$ the inner automorphism. Then, for all $i\le j$, we have
\begin{equation}\label{main_equation}
	ae_{ij}a^{-1}=\sum_{l,m}a_{li}b_{jm}e_{lm}.
\end{equation}
In particular, in order to have $\varphi_a$ a graded isomorphism, we necessarily have the following condition. If $\deg e_{ij}\ne\deg e_{lm}$, then necessarily $a_{li}b_{jm}=0$. In particular, if $\deg e_{lm}\ne1$, then $a_{lm}=0$. These computations forces $G\text{-}\deg a=1$. On the other hand, if $a$ is homogeneous matrix in the $G$-grading, then $axa^{-1}$ is a homogeneous element in the grading for any homogeneous $x\in UT_n$. Denote $$H_1^\Gamma=\{\varphi_a\mid a\in UT_n\textit{ invertible and homogeneous of degree $1$}\}.$$ We proved
\begin{Prop}\label{ref_to_def}
	Let $UT_n$ be endowed with any $G$-grading. Then
	$$
		\text{Aut}_G(UT_n,G)\simeq H_1^\Gamma.
	$$
\end{Prop}

Let $J$ be the set of strict upper triangular matrices (the Jacobson radical of $UT_n$). Let $e_{i:m}=e_{i,i+m}\in UT_n$ and $j\in J^{m+1}$, then for any $\psi\in\text{Aut}(UT_n)$, we have $\psi(e_{i:m}+j)\in\langle e_{i:m}\rangle+J^{m+1}$. As a conclusion, for any grading on $UT_n$, and for all $i,m$, there exists $j_{i,m}\in J^{m+1}$ such that $e_{i:m}+j_{i,m}$ is homogeneous in the grading. Other consequence is that we can't have $\psi\in\text{Aut}(\Gamma)$ unless $\psi\in\text{Stab}(\Gamma)$. This is the statement of
\begin{Prop}\label{prop3}
	Let $UT_n$ be endowed with any $G$-grading $\Gamma$. Then $\text{Aut}(\Gamma)=\text{Stab}(\Gamma)$ and $W(\Gamma)=1$.
\end{Prop}

We proceed now to compute the diagonal subgroup. First, note that in order to have $\varphi_a\in\text{Diag}(UT_n,\eta)$, it is necessary that $\varphi_a(e_{ij})=\lambda_{ij}e_{ij}$, for some (non-zero) $\lambda_{ij}\in K$. From equation (\ref{main_equation}), this condition holds only if $a$ is diagonal. On the other side,  let $a=\text{diag}(a_1,a_2,\ldots,a_n)$ be a diagonal invertible matrix. Then
$$
	\varphi_a(e_{ij})=a_ia_j^{-1}e_{ij},\text{ for all $i\le j$}.
$$
So we must have $a_ia_j^{-1}=a_ka_l^{-1}$ whenever $\deg e_{ij}=\deg e_{kl}$. This combinatorial restriction can be easily solved if we consider the elementary grading given by a sequence $a\in G^n$ (we repeat the construction): a $G$-grading on $UT_n$ is elementary if and only there exists a sequence $(g_1,g_2,g_3,\ldots,g_{n})\in G^n$ such that $\deg e_{ij}=g_ig_j^{-1}$, where we can always assume $g_1=1$. Choose a homomorphism of groups $\chi:G\to(K^\ast,\cdot)$ (we also call $\chi$ a character of $G$; and we denote $\widehat G=\{\chi:G\to(K^\ast,\cdot)\text{ group homomorphism}\}$), and let $a_\chi=\text{diag}(\chi(1),\chi(g_2),\ldots,\chi(g_{n}))$. Clearly $\varphi_{a_\chi}$ will satisfy the conditions for being an automorphism of graded algebras; but not every $\varphi_a\in\text{Diag}(UT_n,\eta)$ is given by such form. We shall introduce the notion of \textit{Universal group grading} of a grading, given in \cite{EldKoc}.

Let $\Gamma':A=\oplus_{g\in G}A_g$ be a $G$-graded algebra. Let $X=\text{Supp}\,\Gamma'$ and consider the free associative group $F\langle X\rangle$ generated by $X$. Consider the relations $R=\{ab^{-1}c^{-1}\mid 0\ne A_bA_c\subset A_a,a,b,c\in\text{Supp}\,A\}$. The \textit{Universal group grading} of $\Gamma'$ is the group $U(\Gamma')=F\langle X\mid R\rangle=F\langle X\rangle/R^{F\langle X\rangle}$. Note that $A$ inherits naturally a $U(\Gamma')$-grading, and this grading is equivalent to $\Gamma'$. We remark also that $\text{Supp}\,\Gamma'\subset U(\Gamma')$.

Now, clearly every element of $\text{diag}(\Gamma)$ results of some $\chi\in\widehat{U(\Gamma)}$. This proves

\begin{Prop}
	Let $UT_n$ be endowed with an elementary $G$-grading given by a sequence $a=(1,g_2,\ldots,g_{n})\in G^n$. Then
	$$
		\text{Diag}(UT_n,a)=\{\varphi_{a_\chi}\mid\chi\in\widehat{U(\Gamma)},a_\chi=\text{diag}(1,\chi(g_2),\ldots,\chi(g_{n}))\}.
	$$
\end{Prop}

\noindent We summarize all the computations of this section
\begin{Theo}
	Consider any elementary $G$-grading $\Gamma$ on $UT_n$, and let $a=(1,g_2,\ldots,g_{n})\in G^n$ be the sequence defining the grading. Then
	\begin{eqnarray*}
		\text{Aut}(\Gamma)=\text{Stab}(\Gamma)&=&\{\varphi_a\mid a\in UT_n\text{ invertible, homogeneous of degree $1$}\},\\
		W(\Gamma)&=&1,\\
		\text{Diag}(\Gamma)&=&\{\varphi_{a_\chi}\mid\chi\in\widehat{U(\Gamma)},a_\chi=\text{diag}(1,\chi(g_2),\ldots,\chi(g_{n}))\}.
	\end{eqnarray*}

\end{Theo}

\noindent\textbf{Graded involutions.} Now, we deal with the graded involutions. The graded involutions were already computed in \cite{VaZa2009}, under certain conditions, but here we cite the validity of the result despite of these restrictions. We will assume until the end of this section that the characteristic of $K$ is not $2$.

Recall that for an associative algebra $A$, an additive map $\ast:A\to A$ is called an involution if $(a^\ast)^\ast=a$, for all $a\in A$; and $(ab)^\ast=b^\ast a^\ast$, for all $a,b\in A$. We call $\ast$ an involution of first kind if $\ast$ is the identity in the center of $A$, and of second kind otherwise. Since the center of $UT_n$ is the scalar multiples of the identity then being an involution of first kind is equivalent of requiring $\ast$ to be linear. We shall study only the case of linear involutions. For the special case where $A$ have a $G$-grading, we name $\ast$ a \textit{graded involution} if $G\text{-}\deg a^\ast=G\text{-}\deg a$, for all homogeneous $a\in A$.

First, we recall some facts concerning involutions in the ungraded case. For the particular case $A=UT_n$, $t:e_{ij}\mapsto e_{n-j+1,n-i+1}$ is an example of involution. Another example is the \textit{symplectic involution} defined bellow, only when $n$ is even.

Let $n=2m$, and consider the diagonal matrix
$$D=\text{diag}(\underbrace{1,1,\ldots,1}_{\text{$m$ times}},\underbrace{-1,-1,\ldots,-1}_{\text{$m$ times}}),$$
note that $D^{-1}=D$. Let $s:UT_n\to UT_n$ be defined by $s(x)=Dt(x)D$, for all $x\in UT_n$. Then $s$ is indeed an involution and we call $s$ the symplectic involution of $UT_n$.

Now, let $\ast$ be any linear involution on $UT_n$. Then the composition $t\circ^\ast$ is an automorphism of $UT_n$; being every automorphism of $UT_n$ inner, we see that necessarily $x^\ast=At(x)A^{-1}$, for some invertible matrix $A\in UT_n$. Moreover, using the relation $(x^\ast)^\ast=x$, we see that necessarily $A^t=A$ or $A^t=-A$. Conversely, every invertible $A\in UT_n$ satisfying $A^t=A$ or $A^t=-A$ construe an involution if we put $\varphi_A\circ t$.

If $(A,\ast)$ and $(B,\circ)$ are two algebras with involutions, we say that $\ast$ and $\circ$ are equivalents if there exists an isomorphism of algebras $f:A\to B$ such that $f(x^\ast)=f(x)^\circ$, for all $x\in A$. In consideration of involutions on the $UT_n$, it was proved in \cite{VinKoSca} that every involution on $UT_n$ is equivalent either to the canonical $t$, or the symplectic $s$.

Now assume an elementary grading $\Gamma$, given by $\eta=(g_1,\ldots,g_{n-1})\in G^{n-1}$, on $UT_n$, and assume that $\ast$ is a graded involution on $UT_n$.
\begin{Lemma}
	In this case, necessarily $\eta$ is symmetric and $\text{Supp}\,\Gamma$ is commutative.
\end{Lemma}
\begin{proof}
	We know that $\ast=\varphi_a\circ t$, for some $\varphi_a$. Using the discussion before Proposition \ref{prop3}, we see that necessarily $e_{i,i+1}^\ast\in\langle e_{n-i,n-i+1}\rangle+J^2$, hence $g_i=g_{n-i}$, which implies $\eta$ symmetric. Moreover, we have:
	$$
		(e_{i,i+1} e_{i+1,i+2})^\ast=e_{i+1,i+2}^\ast e_{i,i+1}^\ast,
	$$
	hence $g_ig_{i+1}=g_{i+1}g_i$ for all $i$. If $i<j$, and letting $g=\deg e_{i+1,j}=g_{i+1}g_{i+2}\cdots g_{j-1}$, we have
	$$
		(e_{i,i+1}e_{i+1,j}e_{j,j+1})^\ast=e_{j,j+1}^\ast e_{i+1,j}^\ast e_{i,i+1}^\ast,
	$$
	so $g_igg_j=g_jgg_i$. By induction hypothesis, we can assume that $[g,g_i]=1$, which implies $[g_i,g_j]=1$. Since $g_1,g_2,\ldots,g_{n-1}$ generates the support of the grading, we see that $\text{Supp}\,\Gamma$ must be commutative.
\end{proof}

Now, we can repeat exactly all the arguments in \cite{VinKoSca} adapting for the graded case, or repeat the arguments in \cite{VaZa2009} to prove:
\begin{Theo}
	Let $UT_n$ be endowed with an elementary grading $\Gamma$ given by $\eta\in G^{n-1}$. If $UT_n$ admits a graded involution $\ast$, then necessarily $\text{Supp}\,\Gamma$ is commutative, $\eta$ is symmetric and $\ast$ is equivalent either to the canonical involution $t$, or the symplectic involution $s$ (this last case may happen only if $n$ is even).

	Conversely, if $\eta=(g_1,g_2,\ldots,g_{n-1})\in G^{n-1}$ is a symmetric sequence of two-by-two commuting elements, and $a\in UT_n$ is any invertible homogeneous matrix with degree $1$ such that $t(a)=a$ or $t(a)=-a$, then $\ast=\varphi_a\circ t$ is a graded involution for $(UT_n,\eta)$.
\end{Theo}
In particular, $\ast$ is equivalent to $t$ if and only if $e_{1n}^\ast=e_{1n}$; $\ast$ is equivalent to $s$ if and only if $e_{1n}^\ast=-e_{1n}$. Also, if $\ast=\varphi_a\circ t$, then $\ast$ is equivalent to $t$ if and only if $t(a)=a$; $\ast$ is equivalent to $s$ if and only if $t(a)=-a$.

\section{Elementary gradings on the non-associative cases}

We start with the Jordan case. Fix an elementary grading $\eta$ on $\text{UJ}_n$. We compute the graded automorphisms for $\eta$. The involution $t$ is always an self-equivalence of $\text{UJ}_n$; moreover $t$ is a graded automorphism of $(\text{UJ}_n,\eta)$ if and only if $\eta$ is symmetric (see definition after Proposition \ref{ref_to_def}). Also, a map $\varphi_a$ will be a graded automorphism if and only if $\varphi:\text{UJ}_n\to\text{UJ}_n$ is a graded map. This happens if and only if $\varphi_a:(UT_n,\eta)\to(UT_n,\eta)$ is a graded map (hence a graded automorphism in the associative case).

Note that a similar discussion as in the associative case can be given for the diagonal group in the Jordan case, since we have the same combinatorial restriction.

All those discussions prove
\begin{Theo}
	Consider an elementary $G$-grading $\Gamma$ on $\text{UJ}_n$, given by $a=(1,g_2,\ldots,g_{n})\in G^n$, where $G$ is commutative. Then
	\begin{enumerate}
		\renewcommand{\labelenumi}{(\roman{enumi})}
		\item $\text{Aut}(\Gamma)=\langle t\rangle\rtimes H_1^\Gamma$.
		\item If $\eta\ne\text{rev}\,\eta$, then $\text{Stab}(\Gamma)=H_1^\Gamma$ and $W(\Gamma)\simeq\mathbb{Z}_2$.
		\item If $\eta=\text{rev}\,\eta$, then $\text{Stab}(\Gamma)=\langle t\rangle\rtimes H_1^\Gamma$ and $W(\Gamma)\simeq1$.
		\item $\text{Diag}(\Gamma)=\{\varphi_{a_\chi}\mid\chi\in\widehat{U(\Gamma)},a_\chi=\text{diag}(1,\chi(g_2),\ldots,\chi(g_{n}))\}$.
	\end{enumerate}
\end{Theo}

Now, let us study the Lie case. Note that if we fix an elementary grading on $UT_n^{(-)}$, then every map $\psi_s\in G_1$ is a graded automorphism. Previous discussion holds also for the Lie case. The elements of $G_1$ cannot belong to the diagonal group; hence the discussion turns out to be similar in the Jordan case. We obtain
\begin{Theo}
	Consider an elementary $G$-grading $\Gamma$ on $UT_n^{(-)}$, given by $a=(1,g_2,\ldots,g_{n})\in G^n$. Then, if $n>2$
	\begin{enumerate}
		\renewcommand{\labelenumi}{(\roman{enumi})}
		\item $\text{Aut}(\Gamma)\simeq\langle\omega\rangle\rtimes(G_1\times H_1^\Gamma)$.
		\item If $\eta\ne\text{rev}\,\eta$, then $\text{Stab}(\Gamma)\simeq G_1\times H_1^\Gamma$ and $W(\Gamma)\simeq\mathbb{Z}_2$.
		\item If $\eta=\text{rev}\,\eta$, then $\text{Stab}(\Gamma)\simeq\langle\omega\rangle\rtimes(G_1\times H_1^\Gamma)$ and $W(\Gamma)\simeq1$.
		\item $\text{Diag}(\Gamma)=\{\varphi_{a_\chi}\mid\chi\in\widehat{U(\Gamma)},a_\chi=\text{diag}(1,\chi(g_2),\ldots,\chi(g_{n}))\}$.
	\end{enumerate}
	For $n=2$, we have $\text{Aut}(\Gamma)=\text{Stab}(\Gamma)\simeq G_1\times H_1^\Gamma$, $W(\Gamma)=1$ and $\text{Diag}(\Gamma)=\{\varphi_a\mid a=\text{diag}(1,x),x\in K^\ast\}$.
\end{Theo}

\section{MT gradings}

Now, we proceed to compute the graded automorphisms for the MT gradings. Note that the involution $t$ is always a graded automorphism for $\text{UJ}_n$; and $\omega$ is always a graded automorphism for $UT_n^{(-)}$, moreover $t$ and $\omega$ are in the respective diagonal subgroup. We shall investigate when an inner automorphism of type $\varphi_a$ is a graded automorphism for the MT grading on the Jordan case.

The following properties of $\varphi_a$ are straightforward and easy to obtain.
\begin{Prop}\label{propum}
	Let $a\in UT_n$ be an invertible matrix. The following conditions are equivalents:
	\begin{enumerate}
		\renewcommand{\labelenumi}{\roman{enumi})}
		\item $\varphi_a$ maps symmetric elements into symmetric elements; and it maps skew-symmetric elements into skew-symmetric elements,
		\item $\varphi_a\circ\omega=\omega\circ\varphi_a$,
		\item there exists a non-zero scalar $k\in K$ such that $a\omega(a)=\omega(a)a=-k\cdot1$.
	\end{enumerate}
\end{Prop}
\begin{proof}
	$(i)\iff(ii)$. Trivial.

	$(ii)\Rightarrow(iii)$. Let $x\in UT_n$. Then
		$$
			a\omega(x)a^{-1}=\omega(axa^{-1})=\omega(a^{-1})\omega(x)\omega(a).
		$$
		In particular, $(a^{-1}\omega(a^{-1}))x(\omega(a)a)=x$ for all $x\in UT_n$, hence $\varphi_{a\omega(a)}=\varphi_1$. From \cite{Do1994}, this gives $a\omega(a)=\omega(a)a=-k\cdot1$ for some $k\in K^\ast$.

	$(iii)\Rightarrow(ii)$. For all $x\in UT_n$, we have
		$$
			\varphi_a^{-1}\circ\omega\circ\varphi_a\circ\omega(x)=a^{-1}\omega(a\omega(x)a^{-1})a=(-k)x(-k^{-1})=x,
		$$
		hence $\varphi_a^{-1}\circ\omega\circ\varphi_a\circ\omega=1$, proving the implication.
\end{proof}

We take a break to remark the following analogous fact for an invertible $a\in UT_n$. The assertions bellow are equivalents:
\begin{enumerate}
	\renewcommand{\labelenumi}{\roman{enumi})}
	\item $\varphi_a$ maps symmetric elements into skew-symmetric elements; and skew-symmetric elements into symmetric elements,
	\item $\varphi_a\circ\omega=-\omega\circ\varphi_a$,
	\item $\omega(a)a=-a\omega(a)=-k\cdot1$, for some non-zero scalar $k\in K$.
\end{enumerate}

We say that a matrix $a\in UT_n$ is \textbf{$\omega$-invertible} if $a\omega(a)=\omega(a)a=-k\cdot1$, for some non-zero $k\in K$. If $\varphi_a$ is a graded automorphism for a MT-grading, then necessarily $a$ is $\omega$-invertible by previous proposition; conversely every $\varphi_a$ with an $\omega$-invertible homogeneous matrix $a$ of degree $1$ is a graded automorphism of an MT-grading. So, $\omega$-invertible matrices are a relevant class of matrices to consider. Next proposition shows how we can construct $\omega$-invertible matrices, and in particular, the proposition proves that such matrices indeed exists.
\begin{Prop}
	Choose a set of elements in the field $R=\{a_{ij}\mid 1\le i\le j<n,i+j<n\}$, with $a_{ii}\in K^\ast$, for all $i$. Fix any $k\in K^\ast$. Assume that one of the following conditions holds:
	\begin{enumerate}
		\item $n$ is even, or
		\item $n$ is odd and there exists $k_0\in K$ with $k_0^2=k$.
	\end{enumerate}
	Then there exists unique $a\in UT_n$ such that $(a)_{(i,j)}=a_{ij}$, for all $1\le i\le j<n$, $i+j<n$, and such that $a\omega(a)=\omega(a)a=-k\cdot1$.
\end{Prop}
\begin{proof}
	We prove by induction on $n$. The claim is true if $n=1$, under the conditions above. If $n=2$, then necessarily $a=\left(\begin{array}{cc}a_{11}&0\\0&a_{11}^{-1}k\end{array}\right)$. So, suppose $n>2$.

	Consider the set $R_0=\{a_{ij}\mid2\le i\le j<n,i+j<n\}\subset R$ and apply induction hypothesis to obtain conditions on the entries belonging to $UT_{n-2}$ of a matrix $a$, if such $a$ indeed exists. Thus, we completely determine the entries $(a)_{(i,j)}$ where $2\le i\le j\le n-1$. Now, assume $(a)_{(1,j)}=a_{1j}$, for $j=1,2,\ldots,n-1$. Assuming that $a\omega(a)=-k\cdot1$, consider the entries $(a\cdot\omega(a))_{(1,j)}$, for $j=1,2,\ldots,n$. Then there exists unique entries $(a)_{(j,n)}$, where $j=2,3,\ldots,n$, that satisfies all these equations. This proves the proposition.
\end{proof}

Now, denote
$$
	\mathscr{H}^G=\{\varphi_a\mid\text{$a\in UT_n$ is $\omega$-invertible, homogeneous and $G\text{-}\deg a=1$}\}.
$$
Note that if the field $K$ is quadratically closed (that is, $x^2=k$, for $k\in K$, always have solution in $K$), then previous definition coincides with all $\varphi_a$ such that $a\in UT_n$ is homogeneous of degree $1$ and $a^{-1}=\omega(a)$.

It is clear that $\mathscr{H}^G$ consists of all graded automorphism of type $\varphi_a$ for a fixed MT-grading. Also, by Proposition \ref{propum}, every element of $\mathscr{H}^G$ commutes with $\omega$ (and with $t$). With this in mind, we obtain
\begin{Prop}
	Let $\Gamma$ be an MT-grading on $\text{UJ}_n$. Then $\text{Stab}(\Gamma)\simeq\mathscr{H}\times\langle t\rangle$.
\end{Prop}

Now, we proceed to make a remark concerning the Weyl group. Reading the classification of all MT-gradings on $\text{UJ}_n$ (see \cite{pkfy2}), we see that the grading is completely defined by the degrees of $X_{i:1}^\pm$. Also, from the discussion before Proposition \ref{prop3}, every equivalence of $\text{UJ}_n$ can send $X_{i:1}^+$ to $X_{i:1}^+$ or to $X_{i:1}^-$, modulo $J^2$. Also, indeed exists an inner automorphism sending $X_{i:1}^+$ to $X_{i:1}^-$ and $X_{i:1}^-$ to $X_{i:1}^+$. So
\begin{Prop}
	Let $\text{UJ}_n$ be endowed with an MT-grading $\Gamma$ and let $p=\lfloor\frac{n}2\rfloor$. Then $W(\Gamma)\simeq\mathbb{Z}_2^p$.
\end{Prop}
\begin{proof}
	Let $p=\left\lfloor\frac{n}2\right\rfloor$. Choose $\epsilon_1,\epsilon_2,\ldots,\epsilon_p\in\{1,-1\}$ and let $\epsilon=\epsilon_1\epsilon_2\cdots\epsilon_p$. Consider the matrix
	$$
		A=\text{diag}(\epsilon,\epsilon,\ldots,\epsilon,\epsilon_1\epsilon_2\cdots\epsilon_{p-1},\ldots,\epsilon_1\epsilon_2,\epsilon_1,1).
	$$
	The map $\varphi_A$ sends $X_{i:1}^+$ to $X_{i:1}^+$ (and $X_{i:1}^-$ to $X_{i:1}^-$) if $\epsilon_i=1$; and sends $X_{i:1}^+$ to $X_{i:1}^-$ (and $X_{i:1}^-$ to $X_{i:1}^+$) if $\epsilon_i=-1$, so $\varphi_A$ gives rise to a weak-isomorphism of $\text{UJ}_n$ (in the sense given in \cite{EldKoc}). By previous discussion, it is not possible to have more weak-isomorphisms. This proves the proposition.
\end{proof}

Finally, we will compute the Diagonal group of an MT-grading. For this purposes, we know that there exists a distinguished element $t\in G$ of order 2 such that $\deg X_{1:0}^-=t$. The group homomorphism $G\to G/\langle t\rangle$ induces a new grading on $\text{UJ}_n$; this grading is elementary and symmetric, say defined by a sequence $a=(a_1,a_2,\ldots,a_n)\in G^n$. Let $U(a)$ be the Universal group grading of such elementary grading $(\text{UJ}_n,a)$. Every $\varphi_{a_\chi}$ is an element of $\text{Diag}(\Gamma)$, where $\chi\in\widehat{U(a)}$ and $a_\chi=\text{diag}(\chi(a_1),\ldots,\chi(a_n))$; we observe that, being the elementary grading symmetric, $a_\chi$ will indeed be a homogeneous matrix of degree 1 in the original MT-grading. We have already discussed that $t$ is an element of the diagonal group, and the conclusion is clear
\begin{Prop}
	Let $\Gamma$ denote an MT-grading on $\text{UJ}_n$ and let the associated induced elementary grading be defined by $a=(a_1,a_2,\ldots,a_n)\in G^n$ and let $U(a)$ be the respective Universal group grading. Then
	$$
		\text{Diag}(\Gamma)\simeq\langle t\rangle\times\{\varphi_{a_\chi}\mid\chi\in\widehat{U(a)},a_\chi=(\chi(a_1),\ldots,\chi(a_n))\}.
	$$
\end{Prop}

\noindent We summarize all the computations of this section for the Jordan case:

\begin{Theo}
	Consider $\text{UJ}_n$ endowed with a MT-grading $\Gamma$, let the induced (symmetric) elementary grading be given by $a=(a_1,\ldots,a_n)\in G^n$ and $U(a)$ be the respective Universal group grading. Then
	\begin{eqnarray*}
		\text{Stab}(\Gamma)&\simeq&\mathscr{H}\times\langle t\rangle,\\
		W(\Gamma)&\simeq&\mathbb{Z}_2^p,\quad p=\left\lfloor\frac{n}2\right\rfloor,\\
		\text{Diag}(\Gamma)&=&\langle t\rangle\times\{\varphi_{a_\chi}\mid\chi\in\widehat{U(a)},a_\chi=\text{diag}(\chi(a_1),\ldots,\chi(a_n))\}.
	\end{eqnarray*}
\end{Theo}

We write one last remark. Take now a MT-grading on the Lie case $UT_n^{(-)}$. Let $G_S=\{\psi_s\mid s=\text{rev}\,s\}$. We can check directly that a map of type $\psi_s\in G_1$ is a graded automorphism for $UT_n^{(-)}$ if and only if $\psi_s\in G_S$. Moreover, note that every element of $G_S$ commutes with $\omega$. Also, the facts concerning the Weyl group and the diagonal group are analogous in the Jordan case. Hence, we obtain
\begin{Theo}
	Consider a MT-grading $\Gamma$ on the Lie algebra $UT_n^{(-)}$. Then
		\begin{eqnarray*}
			\text{Stab}(\Gamma)&\simeq& G_S\times\mathscr{H}^G\times\langle\omega\rangle,\\
			W(\Gamma)&\simeq&\mathbb{Z}_2^p,\quad p=\left\lfloor\frac{n}2\right\rfloor,\\
			\text{Diag}(\Gamma)&=&\langle\omega\rangle\times\{\varphi_{a_\chi}\mid\chi\in\widehat{U(a)},a_\chi=\text{diag}(\chi(a_1),\ldots,\chi(a_n))\},
		\end{eqnarray*}
	where the induced symmetric elementary grading is given by the sequence $a=(a_1,\ldots,a_n)\in G^n$, and $U(a)$ is the respective Universal group grading.
\end{Theo}

\small{\setstretch{0.9}

}
\end{document}